\documentclass[12pt]{amsart}
\usepackage{amsmath,amssymb,amsthm,amscd}
\numberwithin{equation}{section}

\newtheorem{theo}{Theorem}[section]
\newtheorem{lemm}{Lemma}[section]
\newtheorem{coro}{Corollary}[section]

\def\begeq{\begin{equation}}
\def\endeq{\end{equation}}

\def\lf{\left}
\def\ri{\right}

\begin{document}

\title{An existence and uniqueness result for orientation-reversing harmonic diffeomorphism from $\mathbb{H}_*^n$ to $\mathbb{R}_*^n$}
\author{Shi-Zhong Du$^\dagger $ \&   Xu-Qian Fan$^*$}
\thanks{$^\dagger$ Research partially supported by the National
Natural Science Foundation of China (11101106).
}
\address{The School of Natural Sciences and Humanities,
            Shenzhen Graduate School, The Harbin Institute of Technology, Shenzhen, 518055, P. R. China.}
\email{szdu@hitsz.edu.cn}

\thanks{$^*$ Research partially supported by the National
Natural Science Foundation of China (11291240139).
}
\address{Department of
Mathematics, Jinan University, Guangzhou,
 510632,
P. R. China.}
\email{txqfan@jnu.edu.cn}

\renewcommand{\subjclassname}{%
  \textup{2000} Mathematics Subject Classification}
\date{September 2014}

\begin{abstract}
In this paper, we prove an existence and uniqueness theorem for orientation-reversing harmonic diffeomorphisms from $\mathbb{H}_*^n$ to $\mathbb{R}_*^n$ with rotational symmetry, which is a generalization of the corresponding result for dimension $2$.
\end{abstract}
\maketitle\markboth{Shi-Zhong Du $\&$ Xu-Qian Fan}{Existence and uniqueness result for harmonic diffeomorphism}

\section{Introduction}
From the results in \cite{ta2,rr,cl,cdf}, we know that there is no rotationally symmetric harmonic diffeomorphism between the model spaces $\mathbb{R}^n$ and $\mathbb{H}^n$. Even from $\mathbb{R}_*^n$ to $\mathbb{H}_*^n$,  this is also true \cite{df}. But conversely, from $\mathbb{D}^*$ to $\mathbb{C}^*$, it does not hold \cite{cdf2}, although Heinz \cite{hz} obtained the nonexistence of harmonic diffeomorphism from the
unit disc onto the complex plane. In this paper, we generalize the result \cite{cdf2} to general dimension, to find a rotationally symmetric harmonic diffeomorphism from $\mathbb{H}_*^n$ to $\mathbb{R}_*^n$,  and to prove that this map is unique up to a combination of dilation and rotation of $\mathbb{R}^n$. All of these is related to the question mentioned by  Schoen \cite{rs}, which is about
the existence, or nonexistence, of a harmonic diffeomorphism from the complex plane onto the  hyperbolic unit disc. This question has been extenively studied by many people, see for example \cite{wt,httw,atw,mc,cr,wt2,lr} and the references therein. Partial results are related to the Nitsche's type inequalities, see for example  \cite{n,hz,iko,kd4} and the references therein.

As in \cite{rr,cl}, let us denote
\begin{equation*}
\begin{split}
\mathbb{R}^n&=(\mathbb{S}^{n-1}\times [0,\infty),r^2d\theta^2+dr^2) \text{ and }\\
\mathbb{H}^n&=(\mathbb{S}^{n-1}\times [0,\infty),(f(r))^2 d\theta^2+dr^2),
\end{split}
\end{equation*}
where $f(r)=\sinh r $, $(\mathbb{S}^{n-1},d\theta^2)$ is the $(n-1)$-dimensional sphere, and denote
$$\mathbb{R}_*^n=\mathbb{R}^n\setminus \{0\} \textrm{ and }  \mathbb{H}_*^n=\mathbb{H}^n\setminus \{0\}.$$
These notations are applicable for the whole notes.

We prove first the existence and uniqueness of the following linear ordinary differential equation with the boundary conditions.
\begin{lemm}\label{im1}
For $n\geq 2$, every solution $y(r)$ to the following equation
\begin{equation}\label{e1.2}
       y''+(n-1)\frac{f'}{f}\cdot y'-(n-1)\frac{y}{f^2}=0 \text{ for }  r>0
\end{equation}
satisfying the boundary conditions
\begin{equation}\label{e1.2b}
        \lim_{r\to 0^+}y(r)=+\infty, \lim_{r\to +\infty } y(r)=0 \text{ and } y'<0
\end{equation}
is of the form $y=c\sinh^{1-n}r$ for some positive constant $c$.
\end{lemm}

 From this lemma, we can get the following result.
\begin{theo}\label{ithm1}
For $n\geq 2$, there is an orientation-reversing harmonic diffeomorphism from $\mathbb{H}_*^n$ to $\mathbb{R}_*^n$, moreover, it is unique up to a combination of dilation and rotation of $\mathbb{R}^n$.
\end{theo}

 This paper is organized as follows. In Section 2, we will prove Lemma \ref{im1}.
Theorem \ref{ithm1} will be proved in Section 3.
\section*{Acknowledgments}
 The author(XQ) would like to thank Prof. Luen-fai Tam for his providing useful advice on this topic about ten years ago.

\section{Proof of Lemma \ref{im1}}
Noting that from \cite[page 12]{cl},  one can see that $y=\tanh^{-1}r$ is another solution to equation \eqref{e1.2}  for dimension $2$, which is linearly independent to the solution $\overline{y}=\sinh^{-1}r$. From this fact, one can check that Lemma \ref{im1} holds easily. But for general dimension $n\geq 2$, we did not get a solution which is linearly independent to the solution $\overline{y}=\sinh^{1-n}r$, so we need to use  boundary condition \eqref{e1.2b} to get the uniqueness.

Since $y(r)>0$ for $r>0$, divided by $y$ in \eqref{e1.2}, we can get
  $$
   \frac{y''}{y}+(n-1)\frac{f'}{f}\cdot\frac{y'}{y}-(n-1)\frac{1}{f^2}=0.
  $$
Setting
$$x=\frac{y'}{y} \text{ and } z=f\cdot x,$$ we have
$x'=\frac{y''}{y}-\frac{y'^2}{y^2}$ and $z'=f\cdot x'+f'\cdot x.$ Consequently,
equation \eqref{e1.2} can be rewritten as
  \begin{equation}\label{e2.3}
      x'+x^2+(n-1)\frac{f'}{f}\cdot x-(n-1)\frac{1}{f^2}=0,
  \end{equation}
and then
   \begin{equation}\label{e2.13}
     f\cdot z'=-z^2-(n-2)f'\cdot z+(n-1).
   \end{equation}
Since $\overline{y}=\sinh^{1-n}r$ is a solution to \eqref{e1.2} under condition \eqref{e1.2b}, we can see that $\overline{z}$ is a solution to \eqref{e2.13}, where
$\overline{z}=(1-n)\cosh r.$

    Let us study the property of the solution $z$ to \eqref{e2.13}.
 \begin{lemm}\label{p2.2}
  If $y$ is a solution to \eqref{e1.2} under condition \eqref{e1.2b}, then $z(r)$ is the solution of \eqref{e2.13} and
   $$\lim_{r\to0^+}z(r)=1-n.  $$
 \end{lemm}

 The proof of this result will  appear in the later part of this section.

 \begin{coro}\label{l1.1}
   Suppose $z(r)$ is the same as in Proposition \ref{p2.2}, then
   we can get
      $$
        \lim_{r\to0^+}z^{(2k)}(r)=1-n
      $$
   and
      $$
       \lim_{r\to0^+} z^{(2k+1)}(r)=0
      $$
   for all $k=0, 1, 2, \cdots$.
\end{coro}

\noindent
\begin{proof}
 For simplicity, let us denote $z^{(j)}(0)$ as $\lim_{r\to0^+} z^{(j)}(r)$ for $j=0, 1, \cdots.$ From Proposition \ref{p2.2}, we know that the conclusion is true for $z(0)$. We want to show that $z'(0)=0$.  Taking derivative on both sides of \eqref{e2.13}, by elementary computation, we can get
$$z'(0)=nz'(0),$$
which implies $z'(0)=0$.

Suppose Corollary \ref{l1.1} is true for $k-1$ where $k\geq 1$, we need to show that it is true for $k$. Taking $2k$ derivative on both sides of \eqref{e2.13} and using the facts
$$f^{(2i)}(0)=0, f^{(2i+1)}(0)=1 \text{ and } $$
$$C_{2i+1}^0+C_{2i+1}^2+\cdots+C_{2i+1}^{2i}=2^{2i} \text{ for } i\geq 0\text{ with }$$
$$C_{2s}^0+C_{2s}^2+\cdots+C_{2s}^{2s}=2^{2s-1} \text{ for } s\geq 1,$$
 we can get $z^{(2k)}(0)=1-n$. Similarly, we can prove $z^{(2k+1)}(0)=0$.

 By induction, the corollary holds.
\end{proof}

Now we can prove the following estimation of two solutions to \eqref{e2.13}.
\begin{lemm}\label{l1.2}
 Suppose $z(r)$ is a solution of \eqref{e2.13} and $w=z-\overline{z}$, where $\overline{z}(r)=-(n-1)\cosh
r$, then there exists a positive constant $\delta$ such that
   \begin{equation}\label{e1.5}
      \frac{w(r_0)}{r_0^{n-1}}r^{n-1}\leq w(r)\leq
      \frac{w(r_0)}{r_0^{n+1}}r^{n+1}
   \end{equation}
   for    $0<r_0<r<\delta$.
\end{lemm}

\noindent
\begin{proof}
Since $z$ and $\overline{z}$ are two
solutions of \eqref{e2.13} and $w=z-\overline{z}$, we have
  \begin{equation}\label{e1.6}
       fw'=aw,
  \end{equation}
where
   $$
     a(r)=-[z+\overline{z}+(n-2)f']=-z(r)+f'(r)>0.
   $$
Solving the separable equation \eqref{e1.6}, we can get
   \begin{equation}\label{e1.7}
     w(r)=w(r_0)e^{\int^r_{r_0}\frac{a(\tau)}{f(\tau)}d\tau}.
   \end{equation}
Noting that
   $$
     \lim_{r\to0^+}a(r)=n
   $$
and
   $$
     \lim_{r\to0^+}f(r)/r=1,
   $$
we can get
   $$
    \lim_{r\to0^+}\Big(\frac{a(\tau)}{f(\tau)}\Big)\Big/\Big(\frac{n}{\tau}\Big)=1.
   $$
So there exists a positive constant $\delta>0$, such that for
$0<r_0<\tau<r<\delta$, there holds
   $$
    \frac{n-1}{\tau}\leq\frac{a(\tau)}{f(\tau)}\leq\frac{n+1}{\tau}.
   $$
Substituting into \eqref{e1.7}, we can get \eqref{e1.5}. The
conclusion is drawn.
\end{proof}

We are ready to prove Lemma \ref{im1}.
\begin{proof}[Proof of Lemma \ref{im1}]
As mentioned above, we know that
   $$
    \overline{y}(r)=\sinh^{1-n}r
   $$
is a solution to \eqref{e1.2} satisfying condition \eqref{e1.2b}. If $y$ is also a solution to \eqref{e1.2} and \eqref{e1.2b}, then  Corollary \ref{l1.1} guarantees $\lim_{r\to0^+}w^{(j)}(r)=0$ for $j=0, 1, \cdots.$
   So one can get for any
$\alpha>0$,
   $$
     \lim_{r\to0^+}\frac{w(r)}{r^\alpha}=0.
   $$
Taking $r_0\to0^+$ in Lemma \ref{l1.2}, we can get
   $$
     w(r)=0 \text{ for } 0<r<\delta.
   $$
Then the uniqueness theorem of O.D.E. imples
   $$
     w(r)=0 \text{ for }  r>0.
   $$
That is to say,
$$(\ln \overline{y})'=(\ln y)' \text{ for }  r>0.$$
So $y=c \overline{y}$ for some constant $c>0$. Hence the conclusion is drawn.
\end{proof}

In the rest of this section, we want to prove Lemma \ref{p2.2}. The idea is simple: We  find the lower bound of $z$ first, then get the upper bound, and finally,  compute the limit at $0$.

Now let us estimate the lower bound of $z$.
 For each $r>0$, let us consider an quadratic function
     \begin{eqnarray}\label{defQ}
       Q(x)&=& x^2+(n-1)\frac{f'}{f}\cdot
       x-(n-1)\frac{1}{f^2}
     \end{eqnarray}
 in $x$. Clearly, equation \eqref{e2.3} can be rewritten by
    \begin{equation}\label{e3.17}
      x'=-Q(x).
    \end{equation}
    and the roots of $Q(x)=0$ are given by
     $$
        R_1(r)=\frac{-(n-1)f'-\sqrt{(n-1)^2f'^2+4(n-1)}}{2f}<0
     $$
 and
     $$
        R_2(r)=\frac{-(n-1)f'+\sqrt{(n-1)^2f'^2+4(n-1)}}{2f}>0.
     $$

 We will show that a lower bound for $x$ is $R_1$, that is, $z\geq fR_1$. More precisely, we have
\begin{lemm}\label{l2.1}
 If $y(r)$ is a solution to \eqref{e1.2}\eqref{e1.2b}, then we can get
      $$
       0> x(r)\geq R_1(r)\ \ \mbox{ for all } r>0,
      $$
  or equivalently,
      $$
       Q(x(r))\leq0\ \ \mbox{ for all } r>0.
      $$
      Hence $x(r)$ is increasing for $r>0$ and
      $$\lim_{r\to0^+} x(r)=-\infty.$$
\end{lemm}

\begin{proof} The idea of the proof  is similar to that used in Lemma 2.1 \cite{df}. Assume on the contrary, there exists
$\overline{r}>0$ such that
   $$
     x(\overline{r})<R_1(\overline{r}).
   $$
Setting
   $$
    \Sigma=\{\omega\in(\overline{r},+\infty):  x(r)<R_1(r) \mbox{ holds
    true for all } \overline{r}<r<\omega\},
   $$
it is clear that $\Sigma$ is a closed set in
$(\overline{r},+\infty)$. We shall prove that $\Sigma$ is also
a relative open set in $(\overline{r},+\infty)$ to yield
   $$
     \Sigma=(\overline{r},+\infty)
   $$
by connection of $(\overline{r},+\infty)$. In fact, letting
$\omega_0\in\Sigma$, we have
   $$
     x(r)<R_1(r)
   $$
holds for all $r\in(\overline{r},\omega_0)$. So
   $$
     Q(x(r))>0\ \ \mbox{ for all } r\in (\overline{r},\omega_0).
   $$
Using \eqref{e2.3}, we have $x(r)$ is a strictly monotone decreasing
function in $r\in(\overline{r},\omega_0)$. On the other hand, noting that
$R_1(r)$ is monotone non-decreasing in
$r\in(\overline{r},\omega_0)$, we have
   $$
     x(\omega_0)-R_1(\omega_0)< x(\overline{r})-R_1(\overline{r})=-\delta<0
   $$
 for some positive number $\delta$. By continuity, we have
$\omega_0$ is an interior point of $\Sigma$. So $\Sigma$ is also
relative open in $(\overline{r},+\infty)$. Hence
   $$
     \Sigma=(\overline{r},+\infty).
   $$
Consequently,
   $$
    Q(x(r))>0
   $$
for all $r>\overline{r}$. As a result, $x(r)$ is a strictly monotone
decreasing function in $r\in(\overline{r},+\infty)$. In addition,
by the monotonicity of $R_1(r)$, we have
   \begin{equation}\label{e3.18}
      x(r)-R_1(r)<x(\overline{r})-R_1(\overline{r})=-\delta<0
   \end{equation}
for all $r>\overline{r}$. Using \eqref{e2.3} and the fact $x-R_2<0$, we can get
   \begin{eqnarray}\label{e3.19}\nonumber
      x'&=&-Q(x)\\ \nonumber
        &=&-[x(r)-R_1(r)][x(r)-R_2(r)]\\ \nonumber
        &\leq&\delta[x(r)-R_2(r)]\\
        &\leq&\delta x(r)
   \end{eqnarray}
   for $r>\overline{r}$.
So
    $$
     [e^{-\delta r}x(r)]'\leq0 \text{ for } r>\overline{r}.
    $$
Consequently,
    \begin{equation}\label{e2.20}
     x(r)\leq-C_0e^{\delta r}
    \end{equation}
for some constant $C_0>0$ and $r>\overline{r}$.

  Since $f'/f\to 1$ and $f^{-2}\to 0$ as $r\to+\infty$, by \eqref{e2.3} and \eqref{defQ}, we
  can get
    $$
     Q(x(r))\geq\frac{1}{2}x^2(r)
    $$
  for $r>M$, where $M>\overline{r}$ is a large number. As a result,
     \begin{equation}\label{e2.21}
       x'\leq-\frac{1}{2}x^2 \text{ for } r>M.
     \end{equation}
  Consequently,
     $$
      -(x^{-1})'\leq-\frac{1}{2} \text{ for } r>M.
     $$
  After integrating over $r>M$, we get
     $$
      x(r)\leq\frac{1}{r-M+x^{-1}(M)}\to-\infty
     $$
  as $r\to (M-x^{-1}(M))$. This contradicts the fact that $x(r)$ is well-defined in $(0,+\infty)$. Hence for $r>0$, we have
  $$0>x(r)\geq R_1(r).$$
  From these inequalities, one can get $Q(x)\leq 0$, so $x$ is increasing for $r>0$.
 In addition, condition \eqref{e1.2b} implies  $\ln y(r)\to+\infty$
as $r\to 0^+$, so we can get
    \begin{equation*}
         \liminf_{r\to0^+}x(r)=-\infty.
    \end{equation*}
 Hence  $\lim_{r\to0^+}x(r)=-\infty.$     Therefore the conclusion of the lemma is drawn.
  \end{proof}

   Now we want to get the upper bound for $z(r)$.
\begin{lemm}\label{l2.2}
If $y(r)$ is a solution of \eqref{e1.2}\eqref{e1.2b}, then we can obtain
      $$
        z(r)\leq Z_1
      $$
 for all $r>0$, where
   \begin{eqnarray*}
     Z_1&=&\frac{-(n-2)f'-\sqrt{(n-2)^2f'^2+4(n-1)}}{2}<0
   \end{eqnarray*}
and
   \begin{eqnarray*}
     Z_2&=&\frac{-(n-2)f'+\sqrt{(n-2)^2f'^2+4(n-1)}}{2}>0
   \end{eqnarray*}
are roots of quadratic form
   $$
     \widetilde{Q}(z)= z^2+(n-2)f'\cdot z-(n-1).
   $$
\end{lemm}

\begin{proof}
Similar to the proof of Lemma \ref{l2.1}. Assume on the contrary, there exists
$\tilde{r}\in(0,+\infty)$ such that
    $$
      z(\tilde{r})>Z_1.
    $$
Setting
    $$
     \Sigma=\{\omega\in(\tilde{r},+\infty):  z(r)>Z_1 \mbox{ for all }
     r\in(\tilde{r},\omega)\},
    $$
we want to show that $\Sigma=(\tilde{r},+\infty)$. In fact,
$\Sigma\not=\emptyset$ by continuity. It's also clearly that
$\Sigma$ is a closed subset in $(\tilde{r},+\infty)$. We remains
to show that $\Sigma$ is also relative open in
$(\tilde{r},+\infty)$. Actually, for $\omega_0\in\Sigma$, we have
$z(r)$ is a strictly monotone increasing function in
$r\in(\tilde{r},\omega_0)$ by equation \eqref{e2.13}. On the other hand,
since $Z_1(r)$ is a monotone non-increasing function in
$r\in(\tilde{r},\omega_0)$, we have
   $$
    0>z(\omega_0)>Z_1(\omega_0).
   $$
Consequently, $\omega_0$ is an interior point of $\Sigma$. Hence
$\Sigma=(\tilde{r},+\infty)$.

Now we divide this problem into two cases.

Case one: $n=2$. In this case $Z_1=-1$, so $z(r)>-1$ for $r>\tilde{r}$. Since $z=fx$, one have
$x=zf^{-1}$. So
$$(\ln y)'=zf^{-1}>-f^{-1}$$ for $r>\tilde{r}$. Hence
$$y(r)\geq y(\tilde{r})e^{-\int_{\tilde{r}}^\infty f^{-1}(t)dt}.$$
From this, we can get $\lim_{r\to +\infty}y(r)>0$. This contradicts the boundary condition $\lim_{r\to +\infty}y(r)=0$.

Case two: $n\geq 3$.
  Using equation \eqref{e2.13}, we have $z(r)$ is strictly monotone
  decreasing function in $r\in(0,\tilde{r})$. So
     \begin{equation}\label{e3.25}
         0>z(r)\geq-\beta \text{ for } r>\tilde{r}
     \end{equation}
  for some constant $\beta>0$. As a result,
     \begin{equation}\label{e3.26}
       -z^2(r)+n-1\geq-\beta^2+n-1\equiv-\overline{\beta}.
     \end{equation}
  So it follows from equation \eqref{e2.13} that
     $$
      fz'\geq-(n-2)f'\cdot z-\overline{\beta},
     $$
  or equivalent
     $$
       (f^{n-2}z)'\geq-\overline{\beta}f^{n-3}
     $$
  for all $r>\tilde{r}$. Consequently,
      $$
       f^{n-2}(r)z(r)\geq-\overline{\beta}\int^r_{\tilde{r}}f^{n-3}(\tau)d\tau+f^{n-2}(\tilde{r})z(\tilde{r}),
      $$
  or
      \begin{eqnarray}\label{e3.27}\nonumber
         0>z(r)&\geq&-\overline{\beta}\frac{\int^r_{\tilde{r}}f^{n-3}(\tau)d\tau}{f^{n-2}(r)}
         +\frac{f^{n-2}(\tilde{r})z(\tilde{r})}{f^{n-2}(r)}\to
           0^{-}
      \end{eqnarray}
  as $r\to+\infty$, where we have used
     $$
      \lim_{r\to+\infty}f'(r)=+\infty
     $$
  and L' Hospital's rule to get
     \begin{eqnarray*}
       \lim_{r\to+\infty}\frac{\int^r_{\tilde{r}}f^{n-3}(\tau)d\tau}{f^{n-2}(r)}&=&\lim_{r\to+\infty}
       \frac{f^{n-3}(r)}{(n-2)f^{n-3}(r)f'(r)}\\
       &=&\lim_{r\to+\infty}\frac{1}{(n-2)f'(r)}=0.
     \end{eqnarray*}
 Using \eqref{e3.27} and equation \eqref{e2.13}, we have
    $$
     fz'\geq-(n-2)f'z+\lf(n-\frac{3}{2}\ri),
    $$
 or
    \begin{equation}\label{e3.28}
      (f^{n-2}z)'\geq\lf(n-\frac{3}{2}\ri)f^{n-3}
    \end{equation}
  for $r>K$, $K$ large enough. Since $n\geq3$ and
$\lim_{r\to+\infty}f(r)=+\infty$, integrating over $r>K$, we can get
 \begin{equation}
 \begin{split}
      f^{n-2}(r)z(r)&\geq\lf(n-\frac{3}{2}\ri)\int^r_{K}f^{n-3}
      (\tau)d\tau+f^{n-2}(K)z(K)\\
      &\geq\lf(n-\frac{5}{3}\ri)\int^r_{K}f^{n-3}(\tau)d\tau
 \end{split}
 \end{equation}
for $r>M'$, $M'>K$ large enough. Consequently,
    $$
      z(r)\geq\lf(n-\frac{5}{3}\ri)\frac{\int^r_{\tilde{r}}f^{n-3}(\tau)d\tau}{f^{n-2}(r)}>0
    $$
 provided $r>M'$. This contradicts the assumption $z<0$.

 Combining above results, the lemma is proved.
\end{proof}

\noindent
\begin{coro}\label{c2.2}
 Let $z$ be the same as in Lemma \ref{l2.2}, then   $z(r)$
   is a monotone non-increasing function for $r>0$,.
\end{coro}

\begin{proof} Noting that $Z_1(r)$ is the smaller root of
quadratic form $\widetilde{Q}(z)$ and \eqref{e2.13} can be rewritten
by
  \begin{equation}\label{e2.19}
      fz'=-\widetilde{Q}(z)=-(z-Z_1)(z-Z_2)\leq0,
  \end{equation}
so $z(r)$ is a monotone non-increasing function in $r>0$.
\end{proof}

Now let us prove Lemma \ref{p2.2}.
\begin{proof}[Proof of Lemma \ref{p2.2}]
By Lemma \ref{l2.1} and Lemma \ref{l2.2}, we can get
   $$
     f\cdot R_1(r)\leq z(r)\leq Z_1(r).
   $$
Passing to the limits, we can get
  \begin{equation}\label{e2.20}
     f R_1|_{r\to0^+}\leq\lim_{r\to0^+}z(r) \leq Z_1(0)=1-n,
  \end{equation}
where
  $$
     f\cdot R_1|_{r\to0^+}=-\frac{(n-1)f'(0)+\sqrt{(n-1)^2f'^2(0)+4(n-1)}}{2}.
  $$

  We need to show that $\lim_{r\to0^+}z(r) = Z_1(0)$. Assuming on the contrary, by \eqref{e2.20}, we have
   $$
    \lim_{r\to0^+}z(r)<Z_1(0).
   $$
By continuity, there exist small constants $r_0>0$ and $\kappa>0$
such that
  \begin{equation}
     z(r)<Z_1(r)-\kappa
  \end{equation}
for all $0<r<r_0$. Substituting into \eqref{e2.19}, we get
   \begin{equation}\label{e2.21}
    z'(r)\leq\frac{-\kappa'}{f(r)}
   \end{equation}
for some positive constant $\kappa'$. Noting that there exists a positive constant $C$ such that
  $$
    0<f(r)\leq Cr
  $$
for all $0<r<r_0$, after integrating \eqref{e2.21}, we can get
   $$
   0>z(r)\geq
    z(r_0)+\kappa'\int^{r_0}_r\frac{1}{f(\tau)}d\tau\to+\infty
   $$
as $r\to 0^+$. This is impossible. Hence the lemma is proved.
\end{proof}

\section{Proof of Theorem \ref{ithm1}}

\begin{proof}[Proof of Theorem \ref{ithm1}]
 If $u$ is a rotationally symmetric harmonic map from $\mathbb{H}_*^n$ onto $\mathbb{R}_*^n$, then we can assume $u(r,\theta)=(y(r),\theta)$ up to a rotation of $\mathbb{R}^n$. By (1.2) in \cite{cl} for example, $y(r)$ should satisfy the equation \eqref{e1.2}. Furthermore, if $u$ is an orientation-reversing diffeomorphism, then condition  \eqref{e1.2b} is satisfied.

By Lemma \ref{im1}, equation \eqref{e1.2} with \eqref{e1.2b} has a unique solution up to a dilation. Hence the theorem holds.
\end{proof}

\end{document}